\documentclass{amsart}
\usepackage[margin=2 cm]{geometry}

\usepackage{graphics, amsmath,amssymb, enumerate, comment, url, mathtools}
\usepackage{graphicx}
\usepackage{epsfig}

\newif\ifcolorcomments

\colorcommentstrue
\usepackage{amssymb}
\usepackage{amsmath,amsthm}

\usepackage{colortbl}

\newtheorem*{corollary*}{Corollary}
\newtheorem{theorem}{Theorem}[section]
\newtheorem{lemma}[theorem]{Lemma}
\newtheorem{proposition}[theorem]{Proposition}
\theoremstyle{definition}
\newtheorem{definition}[theorem]{Definition}
\newtheorem{remark}[theorem]{Remark}

\newtheorem{conjecture}[theorem]{Conjecture}

\newcommand{\N}{\mathbb N}

\newcommand{\Q}{\mathbb Q}

\newcommand{\R}{\mathbb R}

\newcommand{\cH}{\mathcal{H}}

\newcommand{\cK}{\mathcal{K}}

\newcommand{\cN}{{\mathcal N}}

\DeclareMathOperator{\dimh}{\dim_H}

\renewcommand{\text}{\textup}

\newcommand{\NPC}[1]{\ignorespaces}

\newif\ifdraft\drafttrue

\def\N{\mathbb N}

\def\Q{\mathbb Q}

\def\R{\mathbb R}

\begin{document}


\title{Approximation by Uniformly distributed sequences}

\author[G. Gonz\'alez Robert]{Gerardo Gonz\'alez Robert}
\address{G. Gonz\'alez Robert,  Department of Mathematical and Physical Sciences,  La Trobe University, Bendigo 3552, Australia. }
\email{G.Robert@latrobe.edu.au}

\author[M. Hussain]{Mumtaz Hussain}
\address{Mumtaz Hussain,  Department of Mathematical and Physical Sciences,  La Trobe University, Bendigo 3552, Australia. }
\email{m.hussain@latrobe.edu.au}

\author[N. Shulga]{Nikita Shulga}
\address{Nikita Shulga,  Department of Mathematical and Physical Sciences,  La Trobe University, Bendigo 3552, Australia. }
\email{n.shulga@latrobe.edu.au}

\author[B. Ward]{ Benjamin Ward}
\address{Benjamin Ward,  {Department of Mathematics, University of York, Heslington, YO10 5DD.} }
\email{benjamin.ward@york.ac.uk, ward.ben1994@gmail.com }

\frenchspacing
\maketitle
\begin{abstract}
We consider approximation properties of real points by uniformly distributed sequences. Under some assumptions on the approximation functions, we prove a Khintchine-type $0$-$1$ dichotomy law. We establish a new connection between uniform distribution and the ubiquity property. Namely, we show that a bound on the discrepancy of the sequence implies the ubiquity property, which helps to obtain divergence results. We further obtain Hausdorff dimension results for weighted sets. The key tools in proving these results are the weighted ubiquitous systems and weighted mass transference principle introduced recently by Kleinbock \& Wang, and  Wang \& Wu respectively.

\end{abstract}

\section{Introduction and main results}

Fix a natural number $n$ and let $\omega=(\boldsymbol{\omega}_{j})_{j\geq 1}$ be a sequence of points in $[0,1]^{n}$ (that is, $\boldsymbol{\omega}_{j}=(\omega_{j,1}, \dots , \omega_{j,n})\in [0,1]^{n}$ for each $j\in\N$). Let $\psi_{i}:\N\to \R_{+}$ be a monotonic decreasing function for each $1\leq i \leq n$ and set $\Psi=(\psi_1,\ldots, \psi_n)$. We wish to analyse the question of the approximation of a point $\mathbf{x} \in [0,1]^{n}$ by elements of our chosen fixed sequence $\omega$ from a metrical point of view. For this purpose, we consider the set 
\begin{equation*}
    W_{\omega}(\Psi):=\left\{\mathbf{x}=(x_1,\ldots,x_n) \in [0,1]^{n}: |x_{i}-\omega_{j,i}|<\psi_{i}(j) \, \quad (1\leq i \leq n) \quad \text{ for i.m. } j\in \N\right\}\, .
\end{equation*}
Basically, this is a set of points in an $n$-dimensional unit cube, which are coordinate-wise  $\psi_i$-approximable by a fixed sequence $\omega$. One can easily present degenerate examples of sequences $\omega$ for which the set $W_{\omega}(\Psi)$ has Lebesgue measure $0$. But what are the conditions on $\omega$ under which we can guarantee this happens? And under which conditions on $\omega$ can we guarantee that $W_{\omega}(\Psi)$ has full Lebesgue measure? 

In this article, we present an approach to tackle such questions and provide some partial results. Surprisingly, the framework of ubiquitous systems can be connected to the setting of approximation with sequences $\omega$ by imposing some natural assumptions on the distribution properties of the sequences $\omega$ with standard definitions of uniform distribution and discrepancy.

We recall a few basic definitions from uniform distribution theory, see for example \cite{KuipersNiederreiter1974}. By a \textit{rectangle} $R$ in $[0,1]^n$ we mean a set of the form
\begin{equation*}
    R=\left\{ \mathbf{x} \in [0,1]^{n}: a_{i}\leq x_{i}<b_{i} \quad (1\leq i \leq n) \, \right\}
\end{equation*}
for some $(a_{1},\dots, a_{n}), (b_{1},\dots, b_{n}) \in [0,1]^{n}$ with $a_{i}<b_{i}$ for each $i=1, \ldots, n$. For any rectangle $R$ and any $N\in\N$, define
\begin{equation*}
    A(R;N,\omega): =\#\{1\leq j\leq N : \boldsymbol{\omega}_{j}\in R\} \,.
\end{equation*}
Denote the Lebesgue measure on $[0,1]^n$ by $\lambda_n$. For $n=1$, we write $\lambda = \lambda_1$.
\begin{definition}[Uniformly distributed sequence]
A sequence $\omega= (\boldsymbol{\omega}_{j})_{j \geq 1 }$ of points in $[0,1]^n$ is a \textit{uniformly distributed sequence on} $[0,1]^{n}$, denoted $\omega$ is a $u.d.s$ on $[0,1]^{n}$, if for any $n$-dimensional rectangle $R\subseteq [0,1]^{n}$ we have
\begin{equation*}
    \lim_{N\to \infty} \frac{ A(R;N,\omega)}{N} = \lambda_n(R)\, .
\end{equation*}
\end{definition}
That is, a sequence $\omega$ is uniformly distributed on $[0,1]^{n}$ if for any rectangle $R$ the sequence spends time in $R$ proportionate to the volume of $R$. 
Much is known about $W_{\omega}(\Psi)$ when $\omega$ is a sequence of independent identically distributed uniform random variables. Initiated by Fan and Wu \cite{FanWu04}, who considered the one-dimensional case, the topic has since been investigated and generalised by numerous authors, see for example \cite{ EkstPers18, FJJS18}. These results strongly depend on the randomness of the sequence $\omega$. 
As a classical example of a uniformly distributed sequence, fix some $\alpha\in \R$ and consider the sequence $\omega=(\{n\alpha\})_{n\in \N}$, where $\{x\}$ denotes the fractional part of $x\in\R$. 
The metric properties of $W_{(\{n\alpha\})_{n\in \N}}(\Psi)$ in one dimension have been studied in \cite{Bugeaud03,   KimRamWang18, Tseng08} and in higher dimensions in \cite{BDGW23b,TKim23, HussWard23, Shapira13}. In this setting, it was shown by Chebyshev (in one dimension) and by Khintchine (in higher dimensions) that for $\alpha\in\R\backslash\Q$ the sequence $(\{n\alpha\})_{n\in \N}$ is well distributed \cite{Khin1948}. Recently, Moshchevitin proved metric properties on well-distributed sequences \cite{Mosh23} (his definition of well-distributed sequence differs from the classic one). In the following section, we begin with a general statement in which we consider no other conditions on $\omega$ other than being uniformly distributed. Recall the following definitions of discrepancy, see for example \cite[Chapter 2]{KuipersNiederreiter1974}. 

\begin{definition}[Discrepancy, Star-discrepancy]
Let $\omega$ be a sequence of points in $[0,1]^n$.
\begin{enumerate}[i.]
    \item The \textit{discrepancy} of $\omega$ is the function $D:\N\to\R_+$, $N\mapsto D_N(\omega)$, given by
    \[
    D_N =
    D_N(\omega)
    :=
    \sup_{R\subseteq [0,1]^n}
    \left| \frac{A(R;N,\omega)}{N} - \lambda_n(R)\right|, 
    \]
    where the supremum is taken over all rectangles $R\subseteq [0,1]^{n}$.
    \item The \textit{star-discrepancy} of $\omega$ is the function $D^*:\N\to\R_+$, $N\mapsto D_N^*(\omega)$, given by
    \[
    D_N^* =
    D_N^*(\omega)
    :=
    \sup_{\{\boldsymbol{0}\}\subset R\subseteq [0,1]^n}
    \left| \frac{A(R;N,\omega)}{N} - \lambda_n(R)\right|, 
    \]
    where the supremum is taken over all rectangles $R\subseteq [0,1]^{n}$ of the form $[0,t_1)\times \cdots [0,t_n)$ for some $t_1,\ldots, t_n> 0$.
\end{enumerate}
\end{definition}
It is well known \cite[Theorem 1.3]{KuipersNiederreiter1974} that for any sequence $\omega$ in $[0,1]^n$ and any $N\in\N$ we have $$D_N^*(\omega)\leq D_N(\omega) \leq 2^n D_N^*(\omega).$$
It is clear, see \cite[Ch. 2 Example 1.2]{KuipersNiederreiter1974}, that $\omega$ is a uniformly distributed sequence in $[0,1]^{n}$ if and only if $D_{N}\to 0$ (and by the above observation also $D_{N}^{*}\to 0$) as $N\to \infty$. To relate the notion of discrepancy to the framework of ubiquitous systems, it is crucial to understand how fast the discrepancy decays to zero. For this reason, we introduce the following definition.
\begin{definition}[$(\cN,v)$-Discrepancy satisfying sequence] \label{d.s.s}
    Let $\omega =  (\boldsymbol{\omega}_{j})_{j\geq 1 }$ be a sequence of points in $[0,1]^{n}$. We say that $\omega$ is a \textit{$(\cN,v)$-Discrepancy satisfying sequence}, which we write as $(\cN,v)$-d.s.s, if there exists a strictly increasing sequence $\cN= (N_{i})_{i\geq 1}$ of positive integers and a monotonic decreasing function $v:\N\to \R_{+}$ with $v(N)\to 0$ as $N\to \infty$ such that
\begin{equation*}
    D_{N_{i}}(\omega)  < v(N_{i}) \quad \text{ and } \quad \limsup_{i\to\infty} N_{i-1}v(N_{i})< 1 \quad (i\in\N).
\end{equation*}
\end{definition}
Clearly, every uniformly distributed sequence is a $(\mathcal{N},v)$-d.s.s, since $D_{N}\to 0$ as $N\to \infty$ one can choose an increasingly sparse sequence $\cN$ satisfying the conditions of Definition~\ref{d.s.s}. However, the sparsity of this sequence will become important in our main theorem. So it is desirable to pick $\cN$ as dense as possible, while still satisfying the conditions of the definition. \par 

In order to state our most general result, recall a function $f$ is $c$-regular with respect to a sequence $\cN=(N_j)_{j\geq 1}$ if, for constant $0<c<1$,
\begin{equation*} 
f(N_{i+1}) \leq c f(N_{i})
\end{equation*}
for all sufficiently large $i$. We prove the following Khintchine-type divergence statement.


\begin{theorem} \label{uniform khintchine}
Let $\Psi=(\psi_{1},\dots,\psi_{n})$ be an n-tuple of monotonic decreasing functions. Suppose that $\omega=( \boldsymbol{\omega}_{i} )_{i\geq 1}$ is a $(\cN,v)$-d.s.s. for $\cN=(  N_{i} )_{i\geq 1}$ and $v:\N\to \R_{+}$. Suppose that $\Psi$ or $v$ is $c$-regular for $\cN$, and that for some vector $\boldsymbol{\tau} =(\tau_1,\ldots,\tau_n)\in\R_+$ verifying
\[
\min\{\tau_1,\ldots, \tau_n\}>0,\quad 
\sum_{j=1}^n \tau_j = 1
\]
every large $N\in\N$ satisfies
\[
\psi_j(N)\leq v(N)^{\tau_j}
\quad
(1\leq j\leq n).
\]
Then
\begin{equation*}
    \lambda_n(W_{\omega}(\Psi))=1 \quad \text{\rm if }\quad \sum_{j=1}^{\infty} v(N_{j})^{-1}\prod_{i=1}^{n}\psi_{i}(N_{j}) = \infty.
\end{equation*}
\end{theorem}

The regularity condition imposed on the approximation functions in Theorem \ref{uniform khintchine} is rather restrictive. It might be possible to relax it, but we believe it does not seem possible within the methods used in this paper.

In \cite{Kiefer1961}, Kiefer showed that for \textit{almost every sequence}\footnote{Here and in the following, when we talk about properties holding for almost every sequence in $[0,1]^n$ we have in mind the probability space of sequences in $[0,1]^n$ along with its Borel $\sigma$-algebra and the product measure induced by the Lebesgue measure on each factor (see \cite[Theorem 8.23]{Kallenberg2021}).} $\omega$ we have
\begin{equation}\label{EQ:Kiefer}
\limsup_{N\to\infty} \;  D_N^*(\omega)\,\sqrt{\frac{2N}{\log\log N}} = 1.
\end{equation}
 Thus, for almost every sequence, the discrepancy cannot converge too fast to zero. A uniform lower bound on the discrepancy was obtained by K. Roth, see for example \cite[Ch. 2 Theorem 2.1]{KuipersNiederreiter1974}. Namely, he proved that if $n\geq2$, then every sequence $\omega$ in $[0,1]^n$ satisfies
\begin{equation}\label{Eq:RothDiscrepancyLowerEstimate}
D_N(\omega)
\gg_n
\frac{(\log N)^{\frac{n-1}{2}}}{N}.
\end{equation}

Moreover, by a theorem of van Aardenne-Ehrenfest \cite[Ch. 2 Corollary 2.1]{KuipersNiederreiter1974}, states that for any sequence $\omega$ 
\[
\limsup_{N\to \infty}\, ND_N^*(\omega)=\infty.
\]
Our second result is the following
\begin{theorem}\label{TEO:UDM1S:Measure:01}
    Let $(\tau_1,\ldots,\tau_n)\in\R_{+}^n$ satisfy
    \[
    \min\{\tau_1,\ldots,\tau_n\}>0 \quad\text{ and }\quad \sum_{j=1}^n\tau_j = 1.
    \]
Let $\Psi=(\psi_{1},\dots,\psi_{n})$ be an $n$-tuple of positive, monotonically decreasing functions on $\N$ such that for every large $N\in\N$
     \begin{equation*}
         \psi_j(N)\leq \left(\frac{\log\log N}{2N}\right)^{\frac{\tau_j}{2}}
\quad
(1\leq j\leq n).
     \end{equation*}
     For almost every uniformly distributed sequence $\omega$, if there is some $M>1$ such that
    \[
    \sum_{j=1}^{\infty}\frac{M^{\frac{3^j}{2}}}{\sqrt{j}} \prod_{i=1}^n \psi_i\left(M^{3^j}\right)=\infty,
    \]
    then $\lambda_n\left( W_{\omega}(\Psi)\right)=1$.
\end{theorem}
Next, we consider a class of sequences characterised by low discrepancy as compared to the 'generic' sequence.
\begin{definition}[Low discrepancy]
    A sequence $\omega$ has \textit{low discrepancy} if
\begin{equation*}
    D_{N}(\omega) \ll \frac{(\log N)^{n}}{N}.
\end{equation*}
\end{definition}

It turns out that the sequence $(M^{3^j})_{j\geq 1}$ in Theorem \ref{TEO:UDM1S:Measure:01} can be replaced with $(M^{j^2})_{j\geq 1}$ when $\omega$ is a low-discrepancy sequence. 

\begin{theorem}\label{TEO:UDM1S:Measure:02}
    Let $(\tau_1,\ldots,\tau_n)\in\R_{+}$ satisfy
    \[
    \min\{\tau_1,\ldots,\tau_n\}>0 \quad\text{ and }\quad \sum_{j=1}^n\tau_j = 1.
    \]
     Let $\omega$ be a low-discrepancy sequence and consider $\Psi=(\psi_{1},\dots,\psi_{n})$ such that for every large $N\in\N$
     \begin{equation*}
         \psi_j(N)\leq D_N^{\tau_j}
\quad
(1\leq j\leq n).
     \end{equation*}
    If for some $M>1$ we have
    \[
    \sum_{j=1}^{\infty}\frac{M^{j^2}}{j^{2n}} \prod_{i=1}^n \psi_i\left(M^{j^2}\right)=\infty,
    \]
    then $\lambda_n\left( W_{\omega}(\Psi)\right)=1$.
\end{theorem}

A similar problem was addressed by Boshernitzan and Chaika in the non-weighted case \cite[Example 8]{BoshernitzanChaika2012}. As they point out, part of the proof of Theorem \ref{TEO:BoshChai} may be obtained from the work of Beresnevich, Dickinson, and Velani on ubiquitous systems \cite{BDV06}. 

\begin{theorem}[{{\cite[Example 1]{BoshernitzanChaika2012}}}]\label{TEO:BoshChai}
    Take $n\in\N$. Let $\psi:\N\to\R_+$ be any positive function and $\Psi=(\psi,\ldots, \psi)$. Almost every sequence $\omega = (\boldsymbol{\omega}_j)_{j\geq 1}$ on $[0,1]^n$ satisfies
    \[
    \lambda_n\left( W_{\omega}(\Psi)  \right)
    =
    \begin{cases}
        0 \quad \text{\rm if } \quad \displaystyle\sum_{j=1}^{\infty} \psi^n(j)<\infty, \\
        1 \quad \text{\rm if } \quad \displaystyle\sum_{j=1}^{\infty} \psi^n(j)=\infty. 
    \end{cases}    
    \]
\end{theorem}
\begin{remark}
There are examples of uniformly distributed sequences $\omega=(\boldsymbol{\omega}_j)_{j\geq 1}$ for which there is some positive function $\psi:\N\to\R_+$ satisfying
\[
\sum_{j=1}^{\infty} \psi^n(j)=\infty
\quad\text{ and }\quad
\lambda_n\left( W_{\omega}(\psi, \ldots, \psi) \right)< 1.
\]
In fact, by Kurzweil's Theorem on inhomogeneous approximation, this is the case when $n=1$, $\alpha$ is an irrational number which is not badly approximable, and $\omega_j=\{j\alpha\}$ for $j\in\N$. 
\end{remark}
We believe the weighted version of Theorem \ref{TEO:BoshChai} is true.
\begin{conjecture}
    For almost every sequence $\omega$ in $[0,1]^n$, for every $n$-tuple of non-increasing functions $\Psi=(\psi_1,\ldots, \psi_n)$ we have
    \[
    \lambda_n(W_{\omega}(\Psi))
    =
    \begin{cases}
        0 \quad \text{\rm if } \quad \displaystyle\sum_{j=1}^{\infty} \prod_{i=1}^n \psi_i(j) < \infty, \\ 
        1 \quad \text{\rm if } \quad \displaystyle\sum_{j=1}^{\infty} \prod_{i=1}^n \psi_i(j) = \infty. \\ 
    \end{cases}
    \]
\end{conjecture}
The convergence half is a consequence of the Borel-Cantelli lemma. The divergence part is expected because of Theorem \ref{TEO:BoshChai} and since we can interpret a uniformly distributed sequence as a generic realisation of a sequence of independent identically distributed uniform random variables.

Finally, we compute the Hausdorff dimension of $W_{\omega}(\Psi)$, $\dimh W_{\omega}(\Psi)$, for certain functions $\Psi$ and sequences $\omega$ with low discrepancy.
\begin{theorem}\label{Jarnik Besicovitch uniform}
    Let $\omega$ be a low discrepancy sequence and let $\boldsymbol{\tau}=(\tau_{1},\dots, \tau_{n})\in \R^{n}_{+}$ satisfy
\begin{equation*}
    \sum_{i=1}^{n}\tau_{i}>1\, .
\end{equation*}
Then, for $\Psi(N)=(N^{-\tau_{1}},\dots, N^{-\tau_{n}})$, we have that
\begin{equation*}
    \dimh W_{\omega}(\Psi) = \min_{1\leq j \leq n}\left\{ \frac{1+\sum_{i:\tau_{j}\geq \tau_{i}}(\tau_{j}-\tau_{i})}{\tau_{j}}\right\}.
\end{equation*}
\end{theorem}

We have the following result for exceptional uniformly distributed sequences.

\begin{theorem}\label{Jarnik Besicovitch uniform2}
    Let $n=1$ and $\omega$ be a uniformly distributed sequence such that 
    \begin{equation*}
        D_{N}(\omega)\ll \frac{1}{N}
    \end{equation*}
    for infinitely many $N\in\N$. Let $\tau>1$. Then, for $\Psi(N)=N^{-\tau}$, we have that
\begin{equation*}
    \dimh W_{\omega}(\Psi) = \frac{1}{\tau}.
\end{equation*}
\end{theorem}
In combination with the Three distance theorem and a result of Khintchine \cite[Satz 1]{Khin26} (that only $\Q$ is singular in one dimension), one can deduce from Theorem \ref{Jarnik Besicovitch uniform2} the following theorem of Bugeaud \cite[Theorem 1]{Bugeaud03}.
\begin{corollary*}[Bugeaud~\cite{Bugeaud03}]
Let $\alpha \in \R\backslash\Q$. Then, for any $\tau>1$,
\begin{equation*}
\dimh W_{(\{n\alpha\})_{n\geq 1}}(\Psi)=\frac{1}{\tau}\, .
\end{equation*}
\end{corollary*}

Some of our results are not optimal. A refinement of our strategy potentially can improve our theorems. However, when $n\geq 2$, Roth's lower bound on the discrepancy \eqref{Eq:RothDiscrepancyLowerEstimate} and our condition on the $(\mathcal{N},v)$-discrepancy imply that the growth of $\mathcal{N}=(N_j)_{j\geq 1}$ is faster than exponential. As a consequence, if we are after a 0-1 dichotomy, we will have to impose some conditions on the approximation functions. A more promising landscape appears when we replace the discrepancy of the sequence $\omega$ with its dispersion \cite[Section 1.1.3.]{DrmotaTichy1997}.

We clarify some notations that will be used throughout. For real quantities $A,B$ and a parameter $t$, we write $A \ll_t B$ if $A \leq c(t) B$ for a constant $c(t) > 0$ that depends on $t$ only (while $A$ and $B$ may depend on other parameters). We write  $A\asymp_{t} B$ if $A\ll_{t} B\ll_{t} A$. If the constant $c>0$ is absolute, we simply write $A\ll B$ and $A\asymp B$. The \textit{Hausdorff dimension} of $F$ is denoted by $\dimh F$, and the \textit{$s$-dimensional Hausdorff measure} is denoted by $\cH^s$.

\section{Setup and necessary definitions}\label{ubiquity}

Local ubiquity for rectangles, as given in \cite{KW23}, lies at the heart of our proofs.
This definition is a generalisation of ubiquity for rectangles as found in \cite{WW19}. The notion of an ``ubiquitous system'' for balls was introduced by Dodson, Rynne, and Vickers \cite{DRV} which was then generalised to the abstract metric space settings in \cite{BDV06}.\par

The notion of ubiquitous systems can be defined in any finite product of locally compact metric spaces equipped with an Ahlfors-regular probability measure in each coordinate metric space. 
However, we only consider $n$-dimensional Euclidean space and so restrict to the relevant framework
\footnote{ We have adapted the notion of ubiquity for rectangles defined in  \cite[Section 2.2]{KW23} into the context of uniformly distributed sequences on $[0,1]^{n}$. To this end, we chose the parameters as follows: $J = \N$; $\beta_j=j$ for all $j\in\N$; for each $j\in \N$, we consider the resonant sets $\mathfrak{R}_{j,i}=\{x_{j,i}\}$ whenever $1\leq i\leq n$, and $\mathfrak{R}_j=\{\boldsymbol{\omega}_j\}$; for the $\kappa$-scaling property, take $\kappa_i=0$ when $1\leq i \leq n$; the exponent the Ahlfors measure is $\delta_i=1$ when $1\leq i \leq n$.
}. 
For any $x\in [0,1]$ and $r>0$, we denote by $B(x,r)\subseteq [0,1]$ the ball with centre $x$ and radius $r$ with respect to the usual metric.
Let $(l_{k})_{k\geq 1}$, $(u_{k})_{k\geq 1}$ be two sequences in $\R_{+}$ such that $u_{k} \geq l_{k}$ with $l_{k} \to \infty$ as $k \to \infty$. 
Define for each $k\in\N$ the set
\begin{equation*}
J_{k}= \{ m \in \N: l_{k} \leq m \leq u_{k} \}.
\end{equation*}

Let $\rho=(\rho_{1}, \dots , \rho_{n})$ be an $n$-tuple of non-increasing functions $\rho_{i}: \R_{+} \to \R_{+}$ such that each $\rho_{i}(x) \to 0$ as $x\to \infty$. 
For any set $A\subset [0,1]$ and $r \in \R_{+}$, we write
\begin{equation*}
\Delta_{i}(A,r)= \bigcup_{a \in A}B_{i}(a,r).
\end{equation*}
and for any point $\mathbf{x}=(x_1,\ldots, x_d)\in [0,1]^n$ we define
\begin{equation*}
\Delta(\mathbf{x},\rho(r))= \prod_{i=1}^{n}  \Delta_{i}(\{x_i\},\rho_{i}(r)).
\end{equation*}
With this notation, we can rewrite $W_{\omega}(\Psi)$ as
\[
W_{\omega}(\Psi)
=
\limsup_{j\to\infty} \Delta\left( \boldsymbol{\omega}_j , \Psi(\beta_{j}) \right).
\]

\begin{definition}[Locally ubiquitous sequence]\rm
We say that the sequence $\omega$ is \textit{locally ubiquitous with respect to $(l_k)_{k\geq 1}$, $(u_k)_{k\geq 1}$, and $\rho$} if there exists a constant $c>0$ such that for any ball $B \subset [0,1]^n$ and all sufficiently large $k \in \N$ we have
\begin{equation*}
\mu \left( B \cap \bigcup_{j \in J_{k}}\Delta( \boldsymbol{\omega}_j, \rho(u_{k})) \right) \geq c \mu(B).
\end{equation*}
\end{definition}

We provide a version of a result by Kleinbock and Wang \cite{KW23} which is crucial for our result. 
\begin{theorem}[{\cite[Theorem 2.5]{KW23}}] \label{KW ambient measure}
Let $\omega$ be a locally ubiquitous sequence with respect to $(l_k)_{k\geq 1}$, $(u_k)_{k\geq 1}$, and $\rho$.
For each $1\leq i \leq n$, consider a function $\psi_{i}:\N\to\R_{+}$ and write $\Psi=(\psi_1, \ldots,\psi_n)$.
Suppose that
    \begin{enumerate}
        \item[\rm (I)] for each $1\leq i \leq n$,  $\psi_{i}$ is decreasing,
        \item[\rm (II)] for each $1\leq i \leq n$,  $\psi_{i}(r) \leq \rho_{i}(r)$ for all $r\in\R_{+}$ and $\rho_{i}(r)\to 0$ as $r\to \infty$,
        \item[\rm (III)] either $\rho_{i}$ is $c$-regular on $( u_{k})_{k\geq 1}$ for all $1\leq i \leq n$ or $\psi_{i}$ is $c$-regular on $ ( u_{k})_{k\geq 1}$ for all $1\leq i \leq n$ for some $0<c<1$.
    \end{enumerate}
    Then,
    \begin{equation*}
        \lambda_n( W_{\omega} (\Psi))=\lambda_n([0,1]^n) \quad \text{\rm if }\quad  \quad \sum\limits_{k=1}^{\infty}\prod_{i=1}^{n} \frac{\psi_{i}(u_{k})}{\rho_{i}(u_{k})}  \, = \infty.
    \end{equation*}
\end{theorem}

\section{Proofs of the main results}

\subsection{The ubiquity property}
In this subsection, we formulate a key statement that will allow us to construct ubiquitous systems.

\begin{proposition} \label{ubiquity statement}
    Let $\omega$ be a $(\cN,v)$-d.s.s. . Let $\rho=(\rho_{1},\dots,\rho_{n})$ be an $n$-tuple of functions $\rho_{i}:\cN \to \R_{+}$ such that
    \begin{enumerate}
        \item[i)] each $\rho_{i}(N)\to 0$ as $N\to \infty$ ,
        \item[ii)] $\prod_{i=1}^{n}\rho_{i}(N)=v(N)$ for all $N\in\cN$ .
    \end{enumerate}
    Then there exists constant $c>0$ such that, for any ball $B\subset [0,1]^{n}$, there exists sufficiently large $k_{0}\in\N$ such that for all $k\geq k_{0}$
    \begin{equation*}
        \lambda_n\left( B\cap \bigcup_{N_{k-1}<j\leq N_{k}}\Delta\left(\boldsymbol{\omega}_{j},\rho(N_{k})\right)\right) \geq c\lambda_n(B).
    \end{equation*}
\end{proposition}

\begin{remark}
Proposition \ref{ubiquity statement} provides conditions ensuring that any sequence $\omega$ which is $(\cN,v)$-d.s.s. is locally ubiquitous with respect to $(l_j)_{j\geq 1}=(N_{j-1})_{j \geq 1}$ with $N_0=0$, $(u_j)_{j\geq 1}=\cN$, and $\rho$. In this case, we say that $\omega$ is locally ubiquitous with respect to $\cN$ and $\rho$.
\end{remark}

In order to show Proposition~\ref{ubiquity statement}, we prove several preliminary statements. The first of which is a straightforward statement on the covering of a rectangle by smaller disjoint rectangles. The proof is left as an exercise. 
\begin{proposition}\label{Prop:AlmostCover}
Let $R\subseteq \mathbb{R}^n$ be a rectangle and $(\mathbf{r}_j)_{j\geq 1}$ a sequence in $\mathbb{R}_{+}^n$ converging to $(0,\ldots,0)$. For every $\varepsilon>0$ there exists $K=K(\varepsilon)\in\N$ such that for all $k\in \N_{\geq K}$ there is a finite set $\mathcal{C}=\mathcal{C}(k)\subseteq R$ verifying
\begin{enumerate}[i.]
\item If $\mathbf{y},\mathbf{w}\in\mathcal{C}$ are different, then $\Delta(\mathbf{y};\mathbf{r}_k)\cap\Delta(\mathbf{w};\mathbf{r}_k)=\varnothing$.
\item For each $\mathbf{y}\in\mathcal{C}$, the rectangle $\Delta(\mathbf{y},\mathbf{r}_k)$ is contained in the interior of $R$.
\item We have
\[
\lambda_n\left( R\setminus \bigcup_{\mathbf{y}\in\mathcal{C}} \Delta(\mathbf{y};\mathbf{r}_{k})\right)< \varepsilon.
\]
\end{enumerate}
\end{proposition}


The following result shows that any ball $B\subseteq [0,1]^{n}$ can be covered, up to a Lebesgue-null set, by rectangles centred on a finite subset of the sequence $\omega$.

\begin{lemma}\label{uniform dirichlet}
Let $\rho$ and $\omega$ satisfy the conditions stated in Proposition~\ref{ubiquity statement}. For any ball $B\subseteq [0,1]^{n}$ we have
\[
\lim_{k\to\infty}
\lambda_n\left( B\cap\bigcup_{j=1}^{N_k} \Delta(\boldsymbol{\omega}_j;\rho(N_k))\right)
=
\lambda_n(B).
\]
\end{lemma}

\begin{proof}
Take any $\varepsilon\in (0,1)$. Let $K$ be the natural number obtained from Proposition \ref{Prop:AlmostCover} applied on $R=B$ and $\mathbf{r}_k =\frac{1}{2} \rho(N_k)$. Take $k\in\N_{\geq K}$ and let $\mathcal{C}=\mathcal{C}(k)$ be the corresponding finite set. For every $\mathbf{y}\in \mathcal{C}$, we have 
\[
\left| 
\frac{A_{N_k}\left(\omega;\Delta\left(\mathbf{y}; \frac{\rho(N_k)}{2}\right) \right)}{N_k} -
 \lambda_n\left(\Delta\left(\mathbf{y}; \frac{\rho(N_k)}{2} \right)\right)
\right|
\leq 
D_{N_k} < v(N_k).
\]
Hence, since
\[
 \lambda_n\left(\Delta\left(\mathbf{y}; \frac{\rho(N_k)}{2} \right)\right) = v(N_k),
\]
we have
\[
A_{N_k}\left(\omega;\Delta\left(\mathbf{y}; \frac{\rho(N_k)}{2}\right) \right)\geq 1\, .
\]
So, for each $\mathbf{y}\in \mathcal{C}$, there exists 
\[
\omega(\mathbf{y})
\in 
\{\boldsymbol{\omega}_1,\ldots, \boldsymbol{\omega}_{N_k}\}\cap \Delta\left(\mathbf{y}; \frac{\rho(N_k)}{2} \right)\, .
\]
Fix one such point $\omega(\mathbf{y})$ for each $\mathbf{y}\in\mathcal{C}$. Note that 
\[
\Delta\left(\mathbf{y}; \frac{\rho(N_k)}{2} \right)
\subseteq
\Delta\left(\omega(\mathbf{y}); \rho(N_k) \right).
\]
Therefore, 
\begin{align*}
\lambda_n\left( B\cap\bigcup_{j=1}^{N_k} \Delta(\boldsymbol{\omega}_j;\rho(N_k))\right)
&\geq
\lambda_n\left( B\cap\bigcup_{\mathbf{y}\in \mathcal{C}} \Delta(x(\mathbf{y}) ;\rho(N_k))\right) \\
&\geq
\lambda_n\left( B\cap\bigcup_{\mathbf{y}\in \mathcal{C}} \Delta\left(\mathbf{y} ;\frac{\rho(N_k)}{2} \right)\right) \\
&>
(1-\varepsilon) \lambda_n(B).
\end{align*}
Since $\varepsilon>0$ was arbitrary, we conclude
\[
\liminf_{k\to\infty} \lambda_n\left( B\cap\bigcup_{j=1}^{N_k} \Delta(\boldsymbol{\omega}_j;\rho(N_k))\right)
\geq 
\lambda_n(B).
\]
The result follows, because we have for every $k\in\mathbb{N}$ that
\[
\lambda_n\left( B\cap\bigcup_{j=1}^{N_k} \Delta(\boldsymbol{\omega}_j;\rho(N_k))\right)
\leq 
\lambda_n(B).
\]
\end{proof}

\begin{lemma}\label{standard cover}
For every $\delta,\eta>0$ there exists $K_0=K_0(\delta)\in\N$ such that for all $k\in \mathbb{N}_{\geq K_0}$ we have
\[
\lambda_n\left( B\cap \bigcup_{j=1}^{N_{k-1}} \Delta(\boldsymbol{\omega}_j; \rho(N_{k}))\right)
<
(1+\delta) (1+\eta)^n N_{k-1}v(N_k)\lambda_n(B).
\]
\end{lemma}
\begin{proof}
Let $K\in \N$ be such that for all $k\in\mathbb{N}_{\geq K}$
\[
\begin{cases}
\lambda_n\left( \Delta(B;(1+\eta)\rho(N_k)\right)
< 
\left(1+\frac{\delta}{2}\right) \lambda_n(B), \\
D_{N_{k-1}}
< 
\frac{\delta}{2}
\lambda_n(B).
\end{cases}
\]
Such $K$ exists, because $\rho(N_k)\to 0$ as $k\to\infty$ and $\omega$ is uniformly distributed (and so $D_{N_{k}}\to 0$). Then, for any $k\in\N_{\geq K}$ we have
\[
\left| 
\frac{A_{N_{k-1}}(\omega;\Delta(B;(1+\eta)\rho(N_k)))}{N_{k-1}}
-
\lambda_n\left( \Delta(B;(1+\eta)\rho(N_k) ) \right)
\right|
< \frac{\delta}{2}\lambda_n(B),
\]
so 
\[
A_{N_{k-1}}(\omega;\Delta(B;(1+\eta)\rho(N_k))  
\leq 
\left(1 + \frac{\delta}{2} + \frac{\delta}{2} \right)N_{k-1} \lambda_n(B)
=
(1+\delta)N_{k-1} \lambda_n(B).
\]
Note that if $\Delta(\boldsymbol{\omega}_j; \rho(N_{k})) \cap B\neq \varnothing$, then $\boldsymbol{\omega}_j$ is contained in $\Delta(B;(1+\eta)\rho(N_k))$ and so
\begin{align*}
\lambda_n\left( B\cap \bigcup_{j=1}^{N_{k-1}} \Delta(\boldsymbol{\omega}_j; \rho(N_{k}))\right)
&\leq 
A_{N_{k-1}}\left(\omega ;\Delta(B;(1+\eta)\rho(N_k))\right)\lambda_n\left( \Delta(\boldsymbol{\omega}_{1};\rho(N_k) )\right) \\
&<
(1+\delta) \lambda_n(B) (1+\eta)^n N_{k-1}v(N_k),
\end{align*}
in the last inequality, we have used $\prod_{i=1}^{n}\rho_{i}(N)=v(N)$.
\end{proof}

\begin{proof}[Proof of Proposition~\ref{ubiquity statement}]
Since $\limsup_{k\to\infty} N_{k-1}v(N_{k})<1$, we can pick $k_{0}$ large enough so that $N_{k-1}v(N_{k})<1-\gamma$ for all $k\geq k_{0}$. Now choose small values $c,\varepsilon,\delta,\eta>0$ such that
\begin{equation*}
    1-\gamma < \frac{1-\varepsilon-c}{(1+\delta)(1+\eta)^{n}}\, .
\end{equation*}
 Combining Lemma~\ref{uniform dirichlet} and Lemma~\ref{standard cover} with $\varepsilon, \delta,\eta>0$ chosen as above, for every large $k$ we have that
\begin{align*}
    \lambda_n\left( B \cap \bigcup_{N_{k-1}\leq j \leq N_{k}}\Delta\left( \boldsymbol{\omega}_{j},\rho(N_{k})\right)\right) &\geq \lambda_n\left( B \cap \bigcup_{1\leq j \leq N_{k}}\Delta\left( \boldsymbol{\omega}_{j},\rho(N_{k})\right)\right)-\lambda_n\left( B \cap \bigcup_{1\leq j \leq N_{k-1}}\Delta\left( \boldsymbol{\omega}_{j},\rho(N_{k})\right)\right) \\
    &\geq (1-\varepsilon)\lambda_n(B)-(1+\delta)(1+\eta)^{n}N_{k-1}V(N_{k})\lambda_n(B)\\
    &\geq c\lambda_n(B),
\end{align*}
as required.
\end{proof}
\subsection{Proof of Theorem \ref{uniform khintchine}}
For each $1\leq i \leq n$, define $\rho_i:\N\to\R_{+}$ by $\rho_i(k)=v(k)^{\tau_i}$ for $k\in\N$.
By Proposition \ref{ubiquity statement}, $\omega$ is ubiquitous with respect to $\cN$ and $\rho$. 
By the conditions of Theorem~\ref{uniform khintchine} either $\Psi$ is $c$-regular or $v$ is $c$-regular (which implies that $\rho$ is $\max_{i} c^{\tau_{i}}$-regular).
Thus Theorem \ref{KW ambient measure} is applicable, and so we conclude
\begin{equation*}
    \lambda_n(W_{\omega}(\Psi))=1 \quad \text{\rm if }\quad \sum_{j=1}^{\infty} v(N_{j})^{-1}\prod_{i=1}^{n}\psi_{i}(N_{j}) = \infty.
\end{equation*}

\subsection{Proof of Theorem \ref{TEO:UDM1S:Measure:01}}

Let $\omega$ be a sequence in $[0,1]^n$ satisfying \eqref{EQ:Kiefer}. Take $\varepsilon>0$ and define Let $v:\N\to\R_{+}$ by
\[
v(N)
:=
(1+\varepsilon)
\sqrt{\frac{\log \log N}{2N}}
\quad (N\in\N).
\]
Let $\mathcal{N}=(N_j)_{j\geq 1}$ be given by $N_j=M^{3^j}$, $j\in\N$.
Hence,
\[
v\left( M^{3^j}\right) 
=
(1+\varepsilon)
\frac{j^{1/2}}{M^{3^j/2}}  \sqrt{\frac{\log 3 +\frac{\log\log M}{j}}{ 2 } }
\quad (j\in\N),
\]
so
\[
N_{j-1}v(N_j)
=
M^{3^{j-1}}v\left(M^{3^j}\right) 
\asymp_M 
M^{3^{j-1}} \frac{j^{1/2}}{M^{3^j/2}} 
= 
\frac{j^{1/2}}{(M^{1/6})^{3^j}}  \to 0 \;\text{ as }\; j\to\infty
\]
Moreover, by \eqref{EQ:Kiefer}, $D_{N_j}(\omega)< v(N_j)$ for large $j$.
This shows that $\omega$ is a $(\mathcal{N}, v)$-d.s.s..

For each $i\in\{1,\ldots, n\}$, define the function $\rho_i:\N\to\R$ by
\[
\rho_i(N) = v(N)^{\tau_i}
\quad (N\in\N).
\]
Then, in view of 
\[
\frac{v\left( M^{3^{j+1}}\right) }{v\left( M^{3^{j}}\right)} 
\asymp_M   
\frac{1}{M^{3^j}} \to 0 \quad\text{ as }\quad j \to\infty,
\]
we conclude the $c$-regularity of $\rho$ for any $0<c<1$ with respect to $\cN$. 
We write $l_j=M^{3^{j-1}}$, $j\in\N$, and obtain the result from Theorem \ref{KW ambient measure}.

\subsection{Proof of Theorem \ref{TEO:UDM1S:Measure:02}}
Theorem \ref{TEO:UDM1S:Measure:02} is shown as Theorem \ref{TEO:UDM1S:Measure:01}. In this case, however, we work with the function $\tilde{v}:\N\to\R$ given by
\[
\tilde{v}(N)
=
\frac{(\log N)^n}{N} 
\quad (N\in\N)
\]
and the sequences $l_j = M^{(j-1)^2}$, $N_j=u_j = M^{j^2}$ for all $j\in\N$. In view of
\begin{equation*}
    N_{j-1}\tilde{v}(N_{j})
    =
    M^{(j-1)^{2}}\frac{(j^{2}\log M)^{n}}{M^{j^{2}}}
    =
    j^{2n}(\log M)^{n} M^{-(2j-1)} 
    \to 0
    \quad\text{ as } j\to \infty,
\end{equation*}
we may define $v=C\tilde{v}$ for some positive constant $C=C(\omega)>0$ in such a way that $\omega$ is a $(\cN, v)$-d.s.s for $\cN=(N_j)_{j\geq 1}$.
Finally, the functions 
\begin{equation} \label{ubiquitous system rho}
    \rho_{i}(N)=\left(\frac{(\log N)^{n}}{N}\right)^{\tau_{i}} \quad (1\leq i \leq n)
\end{equation}
are non-increasing and tend to $0$. The theorem now follows from Proposition~\ref{ubiquity statement} and Theorem~\ref{KW ambient measure}. 

\subsection{Proof of Theorems \ref{Jarnik Besicovitch uniform} and \ref{Jarnik Besicovitch uniform2} }
For the lower bound of the Hausdorff dimension and the divergent counterpart of the Hausdorff measure theory, we have the following theorem. It is a particular case of a result by Wang and Wu \cite[Theorem 3.1-3.2]{WW19} on weighted ubiquitous systems.

\begin{theorem}[{\cite{WW19}}] \label{MTPRR}
Let $\omega$ be a sequence in $[0,1]^n$ which is locally ubiquitous with respect to $(l_k)_{k\geq 1}$, $(u_k)_{k\geq 1}$, and $\rho=(\rho^{a_1},\ldots, \rho^{a_n})$ for some function $\rho:\R_{+} \to \R_{+}$ and $(a_{1},\dots, a_{n}) \in \R^{n}_{+}$.
Then, if $\Psi=(\rho^{a_{1}+t_{1}},\dots, \rho^{a_{n}+t_{n}})$ for some $\textbf{t}=(t_{1}, \dots, t_{n}) \in \R^{n}_{+}$,
\begin{equation*}
\dimh W(\Psi) 
\geq 
\min_{A_{i} \in A} 
\left\{ 
\#\cK_{1}
+ 
\#\cK_{2}
+
\frac{\sum\limits_{j \in \cK_{3}}a_{j} -\sum\limits_{j \in \cK_{2}}t_{j}}{A_{i}} 
\right\}=s,
\end{equation*}
where $A=\{ a_{i}, a_{i}+t_{i} , 1 \leq i \leq n \}$ and $\cK_{1},\cK_{2},\cK_{3}$ are a partition of $\{1, \dots, n\}$ defined as
\begin{equation*}
 \cK_{1}=\{ j:a_{j} \geq A_{i}\}, \quad \cK_{2}=\{j: a_{j}+t_{j} \leq A_{i} \} \backslash \cK_{1}, \quad \cK_{3}=\{1, \dots n\} \backslash (\cK_{1} \cup \cK_{2}).
 \end{equation*}
 Furthermore, for any ball $B \subset [0,1]^n$ we have
  \begin{equation} \label{MTPRR_measure}
\cH^{s}(B \cap W_{\omega}(\Psi))=\cH^{s}(B).
\end{equation}
 \end{theorem}


Now let us first prove the upper bounds in Theorems \ref{Jarnik Besicovitch uniform} and \ref{Jarnik Besicovitch uniform2}.
\subsubsection{Upper bound}
For any $K\in\N$, consider the cover
\begin{equation*}
    \bigcup_{j>K}\Delta(\omega_{j},\Psi(j)) \supset W_{\omega}(\Psi)\, .
\end{equation*}
Fix $1\leq k \leq n$. Each rectangle $\Delta(\omega_{j},\Psi(j))$ can be covered by
\begin{equation*}
    \asymp\prod_{i=1}^{n}\max\left\{1, j^{\tau_{k}-\tau_{i}}\right\}=j^{\sum_{i:\tau_{k}>\tau_{i}}(\tau_{k}-\tau_{i})}
\end{equation*}
balls of radius $j^{-\tau_{k}}$. So
\begin{align*}
    \cH^{s}\left(W_{\omega}(\Psi)\right) & \ll \sum_{j>K} j^{\sum_{i:\tau_{k}>\tau_{i}}(\tau_{k}-\tau_{i})} (j^{-\tau_{k}})^{s} \\
    & \leq \sum_{j> K} j^{-s\tau_{k}+\sum_{i:\tau_{k}>\tau_{i}}(\tau_{k}-\tau_{i})} \to 0
\end{align*}
as $K\to \infty$ for any
\begin{equation*}
    s> \frac{1+\sum_{i:\tau_{k}>\tau_{i}}(\tau_{k}-\tau_{i})}{\tau_{k}}\, .
\end{equation*}
Hence
\begin{equation*}
    \dimh W_{\omega}(\Psi) \leq \frac{1+\sum_{i:\tau_{k}>\tau_{i}}(\tau_{k}-\tau_{i})}{\tau_{k}}
\end{equation*}
Since the above argument remains true for each $1\leq k \leq n$, we have the required upper bound. This gives the required upper bound for both Theorem \ref{Jarnik Besicovitch uniform} and \ref{Jarnik Besicovitch uniform2}.

\subsubsection{Lower bound}
        As stated above, for any low discrepancy sequence we can associate the above pair $(\cN,v)$ to ensure that $\omega$ is a $(\cN,v)$-d.s.s and so, by Proposition~\ref{ubiquity statement} $\omega$ is locally ubiquitous with respect to  $\cN$ and $\rho$ as in \eqref{ubiquitous system rho}.
Hence, Theorem \ref{MTPRR} is applicable in the lower bound of Theorem~\ref{Jarnik Besicovitch uniform}. Note that strictly speaking the functions $\Psi$ in Theorem~\ref{Jarnik Besicovitch uniform} should be of the form

    \begin{equation*}
        \Psi(N)=\left(\left(\frac{(\log N)^{n}}{N}\right)^{\tau_{1}},\dots, \left(\frac{(\log N)^{n}}{N}\right)^{\tau_{n}} \right)\, .
    \end{equation*}
    However, it is easily seen that for any $\theta>0$ and any $0<\varepsilon< \theta$ we have  
    \begin{equation*}
        \left(\frac{1}{N}\right)^{\theta}<\left(\frac{(\log N)^{n}}{N}\right)^{\theta}<\left(\frac{1}{N}\right)^{\theta-\varepsilon}
    \end{equation*}
    for $N$ sufficiently large, and so the Hausdorff dimension bound is the same. \par 
    From the conditions of Theorem~\ref{Jarnik Besicovitch uniform2}, we can choose a sequence of natural numbers over which $D_{N}(\omega)\ll N^{-1}$. Take a suitably sparse subsequence of these integers to be $\cN$ such that $N_{k-1}N_{k}^{-1}\to 0$ as $k \to \infty$. Paired with the function $v(N)=cN^{-1}$, where $c$ is the implied constant in the condition of Theorem \ref{Jarnik Besicovitch uniform2}, we have that $\omega$ is a $(\cN,v)$-d.s.s for such pair. 
Hence, $\omega$ is locally ubiquitous with respect to $\cN$ and $\rho=(\rho_1,\ldots, \rho_n)$ determined by
    \begin{equation*}
        \rho_{i}(N)=\rho(N)^{a_{i}}=N^{-a_{i}} \quad (1\leq i \leq n)\, ,
    \end{equation*}
    for vector $\boldsymbol{a}=(a_{1},\dots, a_{n})\in\R_{+}$ such that
    \begin{equation*}
        \sum_{i=1}^{n}a_{i}=1\, .
    \end{equation*}
    \par
    We now consider both lower bound results in tandem. It remains to show the formula given in Theorem \ref{MTPRR} produces the lower bound of Theorem~\ref{Jarnik Besicovitch uniform} and \ref{Jarnik Besicovitch uniform2}.
In order to apply Theorem \ref{MTPRR}, for any $\mathbf{a}=(a_{1},\dots, a_{n})\in \R_{+}$ such that $\sum_{i=1}^{n}a_i=1$ and $\boldsymbol{\tau}=(\tau_{1},\dots,\tau_{n})$ as in Theorem~\ref{Jarnik Besicovitch uniform} and \ref{Jarnik Besicovitch uniform2}, let us define
\begin{equation*}
\textbf{t}=(t_{1},\dots, t_{n})=(\tau_{1}-a_{1}, \dots, \tau_{n}-a_{n})\, .
\end{equation*}
Consider the two cases:
\begin{enumerate}
    \item[i.] Suppose that $\tau_{i}\geq \frac{1}{n}$ for all $1\leq i \leq n$. Then set
    \begin{equation*}
        a_{i}=\frac{1}{n} \quad (1\leq i \leq n)\, .
    \end{equation*}
    Now, if $A=a_{i}$ for any $1\leq i \leq n$ we have that $\cK_{1}=\{1,\dots, n\}$ and so the lower bound formula is $n$ in this case. If $A=\tau_{j}$ for some $1\leq j \leq n$ then
    \begin{equation*}
        \cK_{1}=\emptyset \, , \quad \cK_{2}=\{ i : \tau_{j} \geq \tau_{i}\}\, , \quad \cK_{3}=\cK_{2}^{c}\, ,
    \end{equation*}
    and so
    \begin{align*}
        \dimh W_{\omega}(\Psi)&\geq \min_{1\leq j \leq n}\left\{ \# \cK_{2}+\frac{\sum_{i \in \cK_{2}^{c}} a_{i} -\sum_{i\in \cK_{2}} t_{k}}{\tau_{j}} \right\} \\
        &\geq \min_{1\leq j \leq n}\left\{ \# \cK_{2}+\frac{\sum_{i =1}^{n} a_{i} -\sum_{i\in \cK_{2}} (a_{i}+t_{i})}{\tau_{j}} \right\} \\
        &\geq \min_{1\leq j \leq n}\left\{ \frac{\sum_{i =1}^{n} a_{i} +\sum_{i\in \cK_{2}} (\tau_{j}-(a_{i}+t_{i})}{\tau_{j}} \right\} \\
         &\geq \min_{1\leq j \leq n}\left\{ \frac{1 +\sum_{i:\tau_{j}\geq \tau_{i}} (\tau_{j}-\tau_{i})}{\tau_{j}} \right\}. \\
    \end{align*}

    \item[ii.] Assume that there exists $\tau_{j}$ such that $\tau_{j}<\frac{1}{n}$. Without loss of generality, suppose that 
    \begin{equation*}
        \tau_{1}>\tau_{2}>\dots>\tau_{n}\, .
    \end{equation*}
    We want to choose $1\leq u \leq n$ that solves
    \begin{equation*}
        u \times \widetilde{D} + \sum\limits_{u < i \leq n}\tau_{i}=1
    \end{equation*}
    for some $\widetilde{D}>0$ with $\tau_{u}>\widetilde{D}$. Pick
    \begin{equation*}
        \widetilde{D}=\frac{1-\sum_{u\leq i \leq n}\tau_{i}}{u}
    \end{equation*}
    and note that
    \begin{equation*}
        \tau_{1}>\tau_{2}>\dots>\tau_{u}>\widetilde{D}\geq \tau_{u+1}>\dots > \tau_{n}.
    \end{equation*}
    Hence pick
    \begin{equation*}
        a_{i}=\begin{cases}
            \widetilde{D} \quad (1\leq i \leq u) , \\
            \tau_{i} \quad (u+1\leq i \leq n) .
        \end{cases}
    \end{equation*}
    For $A=a_{i}$ with $1\leq i \leq u$ we have
    \begin{equation*}
        \cK_{1}=\{ 1, \dots , u \}, \quad \cK_{2}=\{u+1, \dots , n\}, \quad \cK_{3}=\emptyset,
    \end{equation*}
    and for $A=a_{i}=\tau_{i}$ with $u+1\leq i \leq n$
    \begin{equation*}
        \cK_{1}=\{1, \dots , i \}, \quad \cK_{2}= \{i+1, \dots , n\}, \quad \cK_{3}=\emptyset.
    \end{equation*}
    In both cases, we have a lower bound formula of $n$. \par 
    For each $A=a_{j}+t_{j}=\tau_{j}$ with $1\leq j \leq u$ we have that
    \begin{equation*}
        \cK_{1}=\emptyset , \quad \cK_{2}=\{ j, \dots , n\}=\{i: \tau_{j}\geq \tau_{i}\}, \quad \cK_{3}=\cK_{2}^{c},
    \end{equation*}
   
\end{enumerate}
 Notice these are the same sets as in the conclusion of case $i)$, and so we obtain the same lower bound, thus the proof of Theorem~\ref{Jarnik Besicovitch uniform} and \ref{Jarnik Besicovitch uniform2} is complete.

\medskip

\noindent{\bf Acknowledgments.} This research is supported by the Australian Research Council Discovery Project 200100994. During the latter stages of the paper, Benjamin Ward was awarded a Leverhulme Early Career fellowship.

\end{document}